\documentclass[letterpaper]{amsart}

\usepackage[utf8]{inputenc}
\usepackage[T1]{fontenc}
\usepackage{lmodern}
\usepackage{mathtools}
\usepackage[pdfusetitle]{hyperref}

\def\arxiv#1{\href{http://arxiv.org/abs/#1}{\texttt{arXiv:#1}}}
\def\doi#1{\href{http://dx.doi.org/#1}{doi:#1}}

\let\phi\varphi

\def\Z{\mathbb Z}
\def\R{\mathbb R}
\def\C{\mathbb C}
\def\H{\mathbb H}
\def\O{\mathbb O}
\def\A{\mathbb A}
\def\m{\mathfrak m}

\def\Rge{\R_{+}}
\def\Sge{S_{+}}
\def\Xge{X_{+}}
\def\Yge{Y_{+}}
\def\Zge{Z_{+}}
\def\bX{\bar X}
\def\bY{\bar Y}
\def\bZ{\bar Z}

\def\pair#1{\langle#1\rangle}
\def\bigpair#1{\bigl\langle#1\bigr\rangle}

\def\Re{\operatorname{Re}}

\DeclareMathOperator{\Span}{span}

\def\ceil#1{\lceil#1\rceil}

\def\bfone{\boldsymbol 1}
\def\bflambda{\boldsymbol\lambda}
\def\ZZ{\boldsymbol Z}
\def\zz{\boldsymbol z}
\def\uu{\boldsymbol u}
\def\aa{\boldsymbol a}
\def\FF{\boldsymbol F}
\def\GG{\boldsymbol G}

\def\HT{H_{T}}

\def\cf{\emph{cf.}}

\theoremstyle{plain}
\newtheorem{theorem}{Theorem}[section]
\newtheorem{proposition}[theorem]{Proposition}
\newtheorem{lemma}[theorem]{Lemma}

\theoremstyle{definition}

\newtheorem{remark}[theorem]{Remark}

\theoremstyle{remark}

\numberwithin{equation}{section}

\begin{document}

\title{Mutants of compactified representations revisited}

\author[M.~Franz]{Matthias Franz}
\address{Department of Mathematics, University of Western Ontario, London, Ont.\ N6A\;5B7, Canada}
\email{mfranz@uwo.ca}

\author[S.~López de Medrano]{Santiago López de Medrano}
\address{Instituto de Matemáticas, Universidad Nacional Autónoma de México, 04510 México, Mexico}
\email{santiago@matem.unam.mx}

\author[J.~Malik]{John Malik}
\address{Department of Mathematics, University of Toronto, 215~Huron Street, Room~HU936, Toronto, Ont.\ M5S\;1A2, Canada}
\email{john.malik@mail.utoronto.ca}

\dedicatory{To the memory of Sam Gitler}
\thanks{M.\,F.\ was supported by an NSERC Discovery Grant
  and J.\,M.\ by an NSERC USRA}


\begin{abstract}
  We show that the mutants of compactified representations
  constructed by Franz and Puppe can be written
  as intersections of real quadrics involving division algebras
  and as generalizations of polygon spaces.
  We also show that these manifolds are
  connected sums of products of spheres.
\end{abstract}

\hypersetup{pdfauthor=\authors}

\maketitle

\section{Introduction}

Let \(n\in\{1,2,4,8\}\),
and let \(T=(S^{1})^{n+1}\) be a torus. 
In~\cite[Sec.~4]{FranzPuppe:2008},
Franz and Puppe defined a ``mutant of a compactified representation''
based on the Hopf fibration~\(S^{2n-1}\to S^{n}\).
We recall the construction in Section~\ref{sec:homeo-Z}.
The mutant is a \((3n+1)\)-dimensional compact orientable \(T\)-manifold~\(Z\)
whose equivariant cohomology
\begin{equation}
  \label{eq:HTZ}
  \HT^{*}(Z;\R) \cong
  \begin{cases}
    A \oplus A[2] \oplus A[2] \oplus A[4] & \text{if \(n=1\),} \\
    A \oplus \m[n-1] \oplus A[2n+2] \oplus A[3n+1] & \text{if \(n>1\)}
  \end{cases}
\end{equation}
is torsion-free over the polynomial ring~\(A=H^{*}(BT;\R)\), but except for~\(n=1\) not free
\cite[Sec.~5]{FranzPuppe:2008}.
(Above, \(\m\lhd A\) denotes the maximal homogeneous ideal
and numbers in square brackets degree shifts.)
This is a rare phenomenon among compact orientable \(T\)-manifolds,
and in fact the mutants were the first examples of this kind,
compare~\cite{Franz:maximal}.
The aim of this note is to show that the mutants
are equivariantly homeomorphic to certain intersections of real quadrics
and to deduce some topological consequences.
In order to define these quadrics, we need to introduce some notation.

Let \(\A\) be a normed real division algebra of dimension~\(n\), hence
isomorphic to either~\(\R\),~\(\C\),~\(\H\) or~\(\O\).
We write \(|{-}|\) for the norm of~\(\A\) and \(\pair{-,-}\) for the associated inner product.
Moreover, let \(\lambda_{0}\),~\dots,~\(\lambda_{n}\) be unit vectors in~\(\A\)
which are symmetric about the origin in the sense that
\begin{equation}
  \label{eq:cond-lambda}
  \pair{\lambda_{k},\lambda_{l}} = \begin{cases}
    \phantom{-}1 & \text{if \(k=l\),} \\
    -\frac{1}{n} & \text{otherwise.}
  \end{cases}
\end{equation}
Such elements exist and are essentially unique, see Remark~\ref{rem:contruct-lambda}.

Now let \(Y\) be the set of solutions of the equations
\begin{align}
  \label{eq:def-Y-1}
  \frac{n}{2}\sum_{k=0}^{n}|Z_{k}|^{2}\lambda_{k} + V W &= 0, \\
  \label{eq:def-Y-2}
  \sum_{k=0}^{n}|Z_{k}|^{2} + |V|^{2} + |W|^{2} &= 1,
\end{align}
where \(\ZZ=(Z_{0},\dots,Z_{n})\in\C^{n+1}\) and \(V\),~\(W\in\A\).
Condition~\eqref{eq:def-Y-1} defines a real algebraic variety in~\(\C^{n+1}\times\A^{2}\)
with an isolated singularity at the origin whose link is \(Y\),
see Proposition~\ref{thm:Y-mf}.
For~\(\A=\C\), the manifold~\(Y\) is among the intersections of real quadrics
studied by Gómez Gutiérrez and López de Medrano in~\cite{GomezLopezDeMedrano:2014}.
The action of~\(T\) on~\(Y\) is induced by the canonical action on the variables~\(Z_{k}\),
\begin{equation}
  (g_{0},\dots,g_{n})\cdot(Z_{0},\dots,Z_{n},V,W) = (g_{0}Z_{0},\dots,g_{n}Z_{n},V,W).
\end{equation}

In the next definition we write
\(\bflambda=(\lambda_{0},\dots,\lambda_{n})\), \(\bfone=(1,\dots,1)\in\A^{n+1}\)
and
\begin{equation}
  \Span(\bflambda,\bfone) = \bigl\{\, V\bflambda + W\,\bfone \bigm| V,W\in\A \,\bigr\}\subset\A^{n+1},
\end{equation}
where the multiplication~\(\A\times\A^{n+1}\to\A^{n+1}\) is defined componentwise.
The ``generalized polygon space''~\(X\) is defined by
\begin{align}
  \label{eq:def-X-1}
  |z_{k}|^{2} + |u_{k}|^{2} &= 1 \qquad\text{for each \(0\le k\le n\)}, \\
  \label{eq:u-in-span-vw}
  \uu &\in \Span(\bflambda,\bfone),
\end{align}
where \(\zz=(z_{0},\dots,z_{n})\in\C^{n+1}\) and \(\uu=(u_{0},\dots,u_{n})\in\A^{n+1}\).
Thus, \(X\) is the intersection of
a product of \(n+1\)~spheres of dimension~\(n+1\)
with a real vector subspace of~\(\C^{n+1}\times\A^{n+1}\)
of dimension~\(4n+2\).
We will see as part of Theorem~\ref{thm:diffeo-X-Y} that \(X\) is also a manifold.
The \(T\)-action on it is defined as for~\(Y\).

Condition~\eqref{eq:u-in-span-vw} is void for~\(\A=\R\),
so that \(X=S^{2}\times S^{2}\)
with the componentwise rotation action of~\(T=S^{1}\times S^{1}\).
For~\(\A=\C\) it means that \(\uu\)
lies on a complex hyperplane of~\(\C^{3}\), and it is not difficult
to verify that in this case
\(X\) is a big polygon space
as introduced in~\cite{Franz:maximal},
see Remark~\ref{rem:lambda-C} below.

\smallskip

We can now state our main results,
see Theorems~\ref{thm:diffeo-X-Y},~\ref{thm:homeo-Z-Y} and~\ref{thm:connected-sums}:

\begin{theorem} \( \)
  \begin{enumerate}
  \item The \(T\)-manifolds~\(X\) and~\(Y\) are equivariantly diffeomorphic.
  \item The mutant~\(Z\) is \(T\)-equivariantly homeomorphic to~\(X\) and~\(Y\).
  \end{enumerate}
\end{theorem}

For~\(\A=\C\) both parts were established in~\cite[Sec.~7]{Franz:maximal}.
Our new contribution lies in a uniform treatment of all cases
including quaternions and octonions.

\begin{theorem}
  As manifolds, \(X\) and~\(Y\) are diffeomorphic
  to connected sums of products of two spheres.
\end{theorem}

An explicit description of the summands is given by formula~\eqref{eq:sum-complex-case}.

\smallskip

We also study ``real versions'' of the manifolds~\(X\),~\(Y\) and~\(Z\) (Section~\ref{sec:real})
as well as generalization to other collections of vectors~\(\bflambda\) (Section~\ref{sec:generalization}).

\section{A diffeomorphism between \texorpdfstring{\(X\) and \(Y\)}{X and Y}}

We start in a slightly more general setting and assume
that \(\lambda_{0}\),~\dots,~\(\lambda_{n}\) are vectors in~\(\R^{n}\)
satisfying~\eqref{eq:cond-lambda}.

\begin{remark}
  \label{rem:contruct-lambda}
  Such vectors
  exist for all~\(n\ge1\) and can be constructed inductively:
  For~\(n=1\) one takes~\(\lambda_{0}=-1\) and~\(\lambda_{1}=1\).
  If \(\mu_{0}\),~\dots,~\(\mu_{n-1}\) 
  is a solution for~\(\R^{n-1}\), then
  \begin{equation}
    \lambda_{k} = \Biggl( \sqrt{1-\frac{1}{n^{2}}}\,\mu_{k},-\frac{1}{n}\Biggr)
    \quad\text{for~\(0\le k\le n-1\),}\qquad
    \lambda_{n} = (0,1)
  \end{equation}
  is one for~\(\R^{n}=\R^{n-1}\times\R\).
  From the following lemma
  one can deduce that any other solution of~\eqref{eq:cond-lambda}
  is of the form~\(A\,\lambda_{0}\),~\dots,~\(A\,\lambda_{n}\) for some~\(A\in O(n)\).
\end{remark}

\begin{lemma}
  \label{thm:lambda-lin}
  Let \(c_{0}\),~\dots,~\(c_{n}\in\R\). Then
  \begin{equation*}
    \sum_{k=0}^{n} c_{k}\lambda_{k} = 0
    \quad\Longleftrightarrow\quad
    c_{0} = \dots = c_{n}.
  \end{equation*}
  In particular, the~\(\lambda_{k}\) span \(\R^{n}\). 
\end{lemma}

\begin{proof}
  In the forward direction the identity
  \begin{equation}
    0 = \bigpair{\lambda_{l}-\lambda_{m},\sum_{k=0}^{n} c_{k}\lambda_{k}}
    = c_{l}-\frac{c_{m}}{n}-c_{m}+\frac{c_{l}}{n}
    = \frac{n+1}{n}(c_{l}-c_{m})
  \end{equation}
  implies that all \(c_{k}\) are equal.
  The backward direction follows from
  \begin{equation}
    \Bigl| \sum_{k=0}^{n}\lambda_{k} \Bigr|^{2}
    = \sum_{k=0}^{n} \pair{\lambda_{k},\lambda_{k}} + \sum_{k\ne l}\pair{\lambda_{k},\lambda_{l}}
    = (n+1) - \frac{n(n+1)}{n} = 0.
    \qedhere
  \end{equation}
\end{proof}

\begin{lemma}
  \label{thm:lambda-proj}
  For any~\(x\in\R^{n}\),
  \begin{equation*}
    x = \frac{n}{n+1}\sum_{k=0}^{n}\pair{\lambda_{k},x}\lambda_{k}.
  \end{equation*}
\end{lemma}

\begin{proof}
  Because
  the \(\lambda_{k}\) span \(\R^{n}\) over~\(\R\) by Lemma~\ref{thm:lambda-lin},
  it suffices to verify the claim for~\(x=\lambda_{l}\),
  where again by Lemma~\ref{thm:lambda-lin} we have
  \begin{equation}
    \sum_{k=0}^{n}\pair{\lambda_{k},\lambda_{l}}\lambda_{k}
    = \lambda_{l} - \frac{1}{n}\sum_{k\ne l}\lambda_{k}
    = \lambda_{l} + \frac{1}{n}\lambda_{l}
    = \frac{n+1}{n}\lambda_{l}.
    \qedhere
  \end{equation}
\end{proof}

We now additionally assume \(n\in\{1,2,4,8\}\) and that
\(\lambda_{0}\),~\dots,~\(\lambda_{n}\) are elements of
the \(n\)-dimensional normed real division algebra~\(\A\).
The scalar product induced by the norm of~\(\A\)
is related to the multiplication
by~\(\pair{a,b}=\Re(\bar a b)\) for all~\(a\),~\(b\in\A\).

Recall that unlike the other division algebras, the octonions are not associative,
but they are still alternative: The subalgebra generated by any two elements~\(a\),~\(b\in\O\)
is associative. By definition, this subalgebra contains \(1\) and therefore
also the conjugates~\(\bar a\) and~\(\bar b\).
A consequence of alternativity is the ``braid law'',
\cf~\cite[Sec.~6.2]{ConwaySmith:2003} or~\cite[Thm.~9.3.1]{EbbinghausEtAl:1993}:
For all~\(a\),~\(b\),~\(c\in\A\) one has
\begin{equation}
  \label{eq:a-b-mu}
  \pair{ab,c} = \pair{b,\bar a c}.
\end{equation}
Good references for octonions and division algebras in general
are \cite{Baez:2002}, \cite{ConwaySmith:2003} and~\cite[Part~B]{EbbinghausEtAl:1993}.

\begin{lemma}
  \label{thm:plane-VW}
  Let \(\uu\in\Span(\bflambda,\bfone)\), say
  \(\uu = V\bflambda + W\,\bfone\) with~\(V\),~\(W\in\A\). Then
  \begin{equation*}
    V = \frac{1}{n+1}\sum_{k=0}^{n}u_{k}\bar\lambda_{k}
    \qquad\text{and}\qquad
    W = \frac{1}{n+1}\sum_{k=0}^{n}u_{k}.
  \end{equation*}
\end{lemma}

In particular, the map~\((V,W)\mapsto\uu\) is injective, so that
\(\Span(\bflambda,\bfone)\) is a real subspace of~\(\A^{n+1}\) of dimension~\(2n\).

\begin{proof}
  By Lemma~\ref{thm:lambda-lin}, we have
  \begin{equation}
    \sum_{k=0}^{n}u_{k} = V\sum_{k=0}^{n}\lambda_{k} + (n+1) W = (n+1)W
  \end{equation}
  and together with alternativity also
  \begin{equation}
    \sum_{k=0}^{n}u_{k}\bar\lambda_{k}
    = \sum_{k=0}^{n}(V\lambda_{k})\bar\lambda_{k} + W \sum_{k=0}^{n}\bar\lambda_{k}
    = V\sum_{k=0}^{n}\lambda_{k}\bar\lambda_{k} = (n+1) V.
    \qedhere
  \end{equation}  
\end{proof}

Let us give names to the defining equations for~\(X\) and~\(Y\):
\begin{align}
  F_{0}\colon \C^{n+1}\times\A^{2} &\to \R,
  & (\ZZ,V,W) &\mapsto \sum_{k=0}^{n}|Z_{k}|^{2} + |V|^{2} + |W|^{2} - 1 \\
  F\colon \C^{n+1}\times\A^{2} &\to \A,
  & (\ZZ,V,W) &\mapsto \frac{n}{2}\sum_{k=0}^{n}|Z_{k}|^{2}\lambda_{k} + V W, \\
  \FF\colon \C^{n+1}\times\A^{2} &\to \R\times\A,
  & (\ZZ,V,W) &\mapsto \bigl(F_{0}(\ZZ,V,W) ,F(\ZZ,V,W) \bigr), \\
  \shortintertext{and for~\(0\le k\le n\)}
  G_{k}\colon \C^{n+1}\times\A^{n+1} &\to \R,
  & (\zz,\uu) &\mapsto |z_{k}|^{2} + |u_{k}|^{2} - 1, \\
  \GG\colon \C^{n+1}\times\A^{n+1} &\to \R^{n+1},
  & (\zz,\uu) &\mapsto \bigl(G_{0}(\zz,\uu),\dots,G_{n}(\zz,\uu)\bigr)
\end{align}
Then \(X=\GG^{-1}(0)\cap\Span(\bflambda,\bfone)\) and \(Y=\FF^{-1}(0)\).
We also extend the \(T\)-actions on~\(X\) and~\(Y\) to the ambient manifolds~\(\C^{n+1}\times\A^{n+1}\)
and~\(\C^{n+1}\times\A^{2}\) in the obvious way, so that \(\FF\) and~\(\GG\) are \(T\)-invariant.

\begin{proposition}
  \label{thm:Y-mf}
  The origin~\(0\in\R\times\A\) is a regular value of~\(\FF\). Thus,
  \(Y\) is a closed \(T\)-submanifold of~\(\C^{n+1}\times\A^{2}\).
\end{proposition}

\begin{proof}
  Because \(F\) is homogeneous, the inverse image of~\(0\in\A\) is a cone.
  It therefore suffices to show that all non-zero solutions of~\(F(\ZZ,V,W)=0\)
  are regular points, for then they will form a submanifold transversal
  to the unit sphere~\(F_{0}^{-1}(0)\). 

  \def\AA{\boldsymbol{A}}
  The derivative of~\(F\) is
  \begin{equation}
    DF(\ZZ,V,W)\cdot(\AA,B,C) =
    n\sum_{k=0}^{n}\pair{Z_{k},A_{k}}\lambda_{k} + B W + V C.
  \end{equation}
  Since \(\A\) is a division algebra, \(DF(\ZZ,V,W)\) is surjective
  whenever \((V,W)\ne(0,0)\). If \(F(\ZZ,0,0)=0\) and \(\ZZ\ne0\), then
  all \(Z_{k}\) are non-zero by Lemma~\ref{thm:lambda-lin}.
  Because the \(\lambda_{k}\) span \(\A\) over~\(\R\), this implies
  that \(DF(\ZZ,0,0)\) is again surjective.
\end{proof}

\begin{theorem}
  \label{thm:diffeo-X-Y}
  \( \)
  \begin{enumerate}
  \item
    \label{p1}
    \(X\) is a closed \(T\)-sub\-man\-i\-fold of~\(\C^{n+1}\times\A^{n+1}\).
  \item
    \label{p2}
    The map~\(\C^{n+1}\times\A^{2}\to\C^{n+1}\times\A^{n+1}\) given by
    \begin{equation*}
      \zz = \sqrt{n+1}\,\ZZ
      \qquad\text{and}\qquad
      \uu = \bar V\bflambda + W\,\bfone
    \end{equation*}
    restricts to a \(T\)-equivariant diffeomorphism~\(Y\to X\).
  \end{enumerate}
\end{theorem}

\begin{proof}
  By Lemma~\ref{thm:plane-VW},
  the map defined in part~\eqref{p2}
  is an equivariant diffeomorphism
  \begin{equation}
    \C^{n+1}\times\A^{2} \to \C^{n+1}\times\Span(\bflambda,\bfone)\subset\C^{n+1}\times\A^{n+1}.
  \end{equation}
  For any~\((\ZZ,V,W)\in\C^{n+1}\times\A^{2}\) we write \((\zz,\uu)\) for its image. 
  We have
  \begin{align}
    |z_{k}|^{2} &= (n+1)|Z_{k}|^{2}, \\
    \label{eq:uk2}
    |u_{k}|^{2} &= |\bar V\lambda_{k}|^{2} + 2\pair{\bar V\lambda_{k},W} + |W|^{2}
    = |V|^{2} + |W|^{2} + 2\pair{\lambda_{k},V W}
  \end{align}
  by~\eqref{eq:a-b-mu}, hence
  \begin{align}
    |z_{k}|^{2} + |u_{k}|^{2} -1
    &= \sum_{l=0}^{n}|Z_{l}|^{2} + |V|^{2} + |W|^{2} -1 \\
    &\qquad + n|Z_{k}|^{2} - \sum_{l\ne k}|Z_{l}|^{2} + 2\pair{\lambda_{k},V W} \\
    &= F_{0}(\ZZ,V,W) + 2\pair{\lambda_{k},F(\ZZ,V,W)}.
  \end{align}
  In matrix notation, this means
  \begin{equation}
    \GG(\zz,\uu)
    = \begin{bmatrix}
      1 & 2\lambda_{0}^{T} \\ \vdots & \vdots \\ 1 & 2\lambda_{n}^{T}
    \end{bmatrix}
    \FF(\ZZ,V,W).
  \end{equation}
  Thus, under the transformation~\((\ZZ,V,W)\mapsto(\zz,\uu)\)
  the restrictions of the equations defining~\(X\)
  to~\(\Span(\bflambda,\bfone)\)
  are related to the equations defining~\(Y\) via some matrix.
  It follows from Lemma~\ref{thm:lambda-lin} that
  this matrix is invertible.
  Together with Proposition~\ref{thm:Y-mf}
  this proves the first claim and at the same time the second.
\end{proof}

\begin{remark}
  \label{rem:lambda-C}
  Assume \(\A=\C\). Then one has \(\lambda_{k}=e^{2\pi i k/3}\mu\) for some~\(\mu\in\C\) of norm~\(1\),
  hence
  in addition to~\(\lambda_{0}+\lambda_{1}+\lambda_{2}=0\)
  also \(\lambda_{0}^{2}+\lambda_{1}^{2}+\lambda_{2}^{2}=0\).
  This shows
  \begin{equation}
    \uu\in \Span(\bflambda,\bfone)
    \quad\Leftrightarrow\quad
    \lambda_{0}u_{0}+\lambda_{1}u_{1}+\lambda_{2}u_{2}=0.
  \end{equation}
  Substituting \(\lambda_{k}u_{k}\) for each~\(u_{k}\), we obtain the condition
  \begin{equation}
    \label{eq:eq-u-111}
    u_{0}+u_{1}+u_{2}=0.
  \end{equation}
  Combining this equation with~\eqref{eq:def-X-1},
  we see that in this case \(X\) is
  the big polygon space~\(X_{1,1}(1,1,1)\)
  introduced in~\cite{Franz:maximal}.

  As discussed in~\cite[Sec.~2]{Franz:maximal}, there is essentially only one
  more \(T\)-manifold one can produce if one replaces \eqref{eq:eq-u-111}
  by some other plane in the \(\uu\)-variables
  (assuming that it intersects the product of spheres
  gives by~\eqref{eq:def-X-1} transversally): This is the manifold~
  \(S^{3}\times S^{3}\times S^{1}\)
  obtained from the equation~\(u_{3}=0\), where the \(3\)~circle factors of~\(T\)
  rotate the \(3\)~spheres through scalar multiplication in the first complex coordinate.
\end{remark}

\section{A homeomorphism with \texorpdfstring{\(Z\)}{Z}}
\label{sec:homeo-Z}

We start with some general considerations.
Let \(M\) be a compact Hausdorff space with an action of~\(T\). The orbit space~\(Q = M/T\) comes
with a partition~\((Q_{K})\) indexed by the closed subgroups~\(K\subset T\),
namely the one that associates to an orbit~\(q\in Q\) the common isotropy group~\(K\)
of its points.

If the projection~\(M\to Q\) admits a section~\(\sigma\), then the map
\begin{equation}
  T\times Q \to M,
  \quad
  (g,q) \mapsto g\cdot\sigma(q)
\end{equation}
is a closed surjection, hence \(M\) is equivariantly homeomorphic to
\begin{equation}
  \label{eq:T-ident-space}
  T\times Q \bigm/ \sim \, ,
\end{equation}
where \((g,q)\sim(g',q')\) if and only if \(q=q'\), say contained in~\(Q_{K}\),
and \(g^{-1}g'\in K\).
Conversely, if we are given a space~\(Q\) together with a partition~\((Q_{K})\),
then we can define a (possibly non-Hausdorff) \(T\)-space by~\eqref{eq:T-ident-space}.

Next we recall the construction of the mutant~\(Z\) from~\cite[Sec.~4]{FranzPuppe:2008}.
Consider the standard action of~\(T=(S^{1})^{n+1}\) on~\(\C^{n+1}\) and on its one-point compactification~\(S\).
We identify \(S\) with the unit sphere~\(S^{2n+2}\) in~\(\C^{n+1}\times\R\),
and the quotient~\(S/T\) with the subset
\begin{equation}
  \Sge = \bigl\{\, (\aa,b)\in S^{2n+2} \bigm| \text{\(a_{k}\in\Rge\) for all~\(k\)} \,\bigr\},
\end{equation}
where we have written \(\Rge\) for the non-negative real numbers. Note that \(\Sge\)
is a ball of dimension~\(n+1\), and its boundary is the \(n\)-sphere
given by the~\((\aa,b)\) with~\(a_{k}=0\) for at least one~\(k\).

Let \(Q\) be a \(2n\)-ball, and let \(\pi\colon\partial Q\approx S^{2n-1}\to S^{n}\approx \partial\Sge\)
be the Hopf fibration. On \(Q\) we introduce a partition by setting \(Q_{1}=Q\setminus\partial Q\),
and \(Q_{K}=\pi^{-1}((\Sge)_{K})\) for~\(K\ne1\).
The \(T\)-space resulting from~\eqref{eq:T-ident-space} is the mutant~\(Z\);
one can show that it is a compact orientable smooth manifold.
Note that only the coordinate subtori of~\(T\) occur as isotropy groups in~\(S\),
hence in~\(Z\).

\begin{remark}
  \label{rem:Z-fixed-points}
  Let \(T=L\times M\) be a decomposition into coordinate subtori
  with~\(L\ne1\), hence \(m=\dim M\le n\).
  Because the Hopf fibration is trivial over any proper subset of~\(S^{n}\),
  there is an equivariant homeomorphism
  \begin{equation}
    Z^{L} \approx S^{2m}\times S^{n-1}
  \end{equation}
  where \(M\) acts in the usual way on the one-point compactification of~\(\C^{m}\).
  In particular, \(Z^{T}\approx S^{0}\times S^{n-1}\), and
  the quotient~\(Z^{L}/M\) is homeomorphic to~\(D^{m}\times S^{n-1}\) for~\(m>0\).
\end{remark}

The \(T\)-manifolds~\(X\) and~\(Y\) are also of the form~\eqref{eq:T-ident-space}.
The quotient~\(X/T\) can be lifted to the subset
\begin{align}
    \Xge &= \bigl\{\, (\zz,\uu)\in X \bigm| \text{\(z_{k}\in\Rge\) for all~\(k\)} \,\bigr\} \\
\shortintertext{and \(Y/T\) to}
  \Yge &= \bigl\{\, (\ZZ,V,W)\in Y \bigm| \text{\(Z_{k}\in\Rge\) for all~\(k\)} \,\bigr\}.
\end{align}
Note that \((\ZZ,V,W)\in \Yge\) is determined by~\((V,W)\):
By~\eqref{eq:def-Y-1},
the~\(Z_{k}\) are essentially the barycentric coordinates of~\(-V W\)
with respect to the affinely independent vectors \(\lambda_{0}\),~\dots,~\(\lambda_{n}\in\A\).
In the same vein, \((\zz,\uu)\in \Xge\) is determined by~\(\uu\).
Consequently, \(\Xge\) is homeomorphic to
\begin{equation}
  \bigl\{\, \uu\in \A^{n+1} \bigm| \text{\(|u_{k}|\le1\) for all~\(k\)} \,\bigr\}
  \cap \Span(\bflambda,\bfone).
\end{equation}
This is the unit ball of~\(\Span(\bflambda,\bfone)\)
with respect to the norm~\(\|{-}\|_{\infty}\) induced by the maximum norm
of~\(\A^{n+1}\), the maximum being taken over all~\(|u_{k}|\);
the boundary sphere is given by the points where at least one coordinate~\(z_{k}\) vanishes.
Thus, \(\Yge\approx \Xge\) is a \(2n\)-ball, and its boundary~\(\partial \Yge\)
is a \((2n-1)\)-sphere given by the points with~\(Z_{k}=0\) for some~\(k\).

Now define
\begin{equation}
\begin{aligned}
  \phi\colon \Yge &\to \Rge^{n+1}\times\A\times\Rge^{2}, \\
  (\ZZ,V,W) &\mapsto \bigl( \ZZ,V W,|V|,|W| \bigr)
\end{aligned}
\end{equation}
and
\begin{equation}
\begin{aligned}
  \psi\colon \phi(\Yge) &\to \Rge^{n+1}\times\R, \\
  \bigl(\ZZ,V W,|V|,|W| \bigr) &\mapsto \Biggl( \frac{\ZZ}{\sqrt{1-2|VW|}} , \frac{|V|-|W|}{\sqrt{1-2|VW|}} \Biggr).
\end{aligned}
\end{equation}
That the last line is well-defined follows from the following observation:

\begin{lemma}
  \label{thm:max-VW}
  Let \((\ZZ,V,W)\in Y\). Then \(|VW|<\frac{1}{2}\).
\end{lemma}

\begin{proof}
  By~\eqref{eq:def-Y-2} we have \(|V|^{2}+|W|^{2}\le 1\).
  The maximum of~\(|VW|\) on the Euclidean unit ball of~\(\A^{2}\) is \(1/2\), and it is assumed at~\(|V|=|W|=1/\sqrt{2}\).
  But then \eqref{eq:def-Y-2} implies \(\ZZ=0\),
  which contradicts \eqref{eq:def-Y-1} since \(V W\ne0\).
\end{proof}

\begin{lemma}
  \label{thm:Q-Hopf}
  \(\phi\colon \partial \Yge\to \phi(\partial \Yge)\)
  is the Hopf fibration~\(S^{2n-1}\to S^{n}\).
\end{lemma}

\begin{proof}
  By~\cite[Sec.~3.1]{Baez:2002}, the restriction of the map
  \begin{equation}
    \pi\colon \A^{2} \to \A\times\Rge^{2},
    \quad
    (V,W) \mapsto \bigl( V W,|V|,|W| \bigr)
  \end{equation}
  to the unit sphere in~\(\A^{2}\) (and its image under~\(\pi\)) is the Hopf fibration~\(S^{2n-1}\to S^{n}\).
  In~\cite{Baez:2002} the usual Euclidean norm is used, but the conclusion remains valid
  for any other norm, in particular for the norm~\(\|{-}\|_{\infty}\) considered above.
  This proves the claim because
  a point~\((\ZZ,V W,|V|,|W|)\in\phi(\Yge)\)
  is determined by~\((V W,|V|,|W|)\),
  analogously to the case of~\(\Yge\) discussed above.
\end{proof}

We need the following elementary observation:

\begin{lemma}
  \label{thm:pq-xy}
  Let \(p\),~\(q\ge0\). The system of equations
  \begin{equation*}
    x^{2}+y^{2} = p
    \quad\text{and}\quad
    xy = q
  \end{equation*}
  is soluble for~\(x\),~\(y\ge0\)
  if and only if \(p\ge 2q\). The solution is unique up to interchanging \(x\) and~\(y\);
  one has \(x=y\) if and only if~\(p=2q\).
\end{lemma}

\begin{proposition}
  \label{thm:p-homeo}
  The map~\(\psi\) is a homeomorphism of~\(\phi(\Yge)\) onto~\(\Sge\).
\end{proposition}

\begin{proof}
  We start by verifying that \(\psi\) maps to~\(\Sge\).
  For~\((\ZZ,V,W)\in \Yge\) we have
  \begin{equation}
    \sum_{k}|Z_{k}|^{2}+\bigl(|V|-|W|\bigr)^{2}
    = \sum_{k}|Z_{k}|^{2} + |V|^{2} - 2|VW| + |W|^{2}
    = 1- 2|VW|
  \end{equation}
  by~\eqref{eq:def-Y-2}, which shows \(\psi(\ZZ,V W,|V|,|W|)\in S_{+}\).

  Next we prove that \(\psi\colon \phi(\Yge)\to\Sge\) is surjective.
  Given \((\aa,b)\in\Sge\), we define
  \begin{equation}
    c = \Bigl| \frac{n}{2}\sum_{k}a_{k}^{2}\,\lambda_{k} \Bigr|
  \end{equation}
  and then
  \begin{equation}
    \label{eq:def-q}
    p = \frac{b^{2}+2c}{1+2c}
    \qquad\text{and}\qquad
    q = \frac{c}{1+2c} ,
  \end{equation}
  the latter equations being equivalent to
  \begin{equation}
    b^{2} = \frac{p-2q}{1-2q}
    \qquad\text{and}\qquad
    c = \frac{q}{1-2q}.
  \end{equation}
  
  We have \(q\ge0\) and \(p-2q=b^{2}/(1+2c)\ge0\).
  By Lemma~\ref{thm:pq-xy} there is a solution to the equations
  \(x^{2}+y^{2}=p\), \(xy=q\), unique up to interchanging \(x\) and \(y\).
  Moreover, \(x=y\) if and only \(p=2q\), that is, if and only if \(b=0\).
  If \(b\ne0\), we choose \(x\) and~\(y\) such that \(x-y\) has the same sign as~\(b\).
  Now we set
  \begin{align}
    \ZZ &= \sqrt{1-2xy}\, \aa, \\
    V &= x, \\
    W &= \begin{cases}
      y & \text{if \(x=0\),} \\
      - \frac{1}{x}\cdot \frac{n}{2}\sum_{k}Z_{k}^{2}\,\lambda_{k}
      & \text{otherwise.}
    \end{cases}
  \end{align}
  (Observe that \(1-2xy=1-2q=q/c>0\).)

  We note that \(|W| = y\).
  This is clear if \(x=0\). Otherwise, we have
  \begin{equation}
    |W| = \frac{1-2xy}{x} \cdot \Bigl| \frac{n}{2}\sum_{k}a_{k}^{2}\,\lambda_{k} \Bigr|
    = \frac{(1-2q)c}{x} = \frac{q}{x} = y.
  \end{equation}

  For~\(x\ne0\), equation~\eqref{eq:def-Y-1} is satisfied by construction.
  If \(x=0\), then \(V W=0\), but also \(q=0\) and \(c=0\), which implies \eqref{eq:def-Y-1}.
  Moreover,
  \begin{align}
    \sum_{k}Z_{k}^{2} + |V|^{2} + |W|^{2}
    &= (1-2xy)\sum_{k}a_{k}^{2} + x^{2} + y^{2} \\
    \notag
    &= (1-2q)(1-b^{2}) + p  = 1,
  \end{align}
  where we have used that \((\aa,b)\) has norm~\(1\).
  Hence \((\ZZ,V,W)\in \Yge\).

  Since \(|VW|=xy\), we have
  \begin{equation}
    \aa = \frac{\ZZ}{\sqrt{1-2|VW|}}.
  \end{equation}
  By construction, \(b\) has the same sign as~\(|V|-|W|=x-y\), and
  \begin{equation}
    \label{eq:id-b2}
    \Biggl(\frac{|V|-|W|}{\sqrt{1-2|VW|}}\Biggr)^{2}
    = \frac{(x-y)^{2}}{1-2xy} = \frac{p-2q}{1-2q} = b^{2}.
  \end{equation}
  Thus \(\psi(\ZZ,V W,|V|,|W|)=(\aa,b)\), and \(\psi\) is surjective.

  It remains to show that \(\psi\) is injective.
  Let \((\aa,b)=\psi(\ZZ,V W,|V|,|W|)\). Then
  \begin{align}
    c &= \frac{1}{1-2|VW|}\Bigl| \sum_{k}Z_{k}^{2}\,\lambda_{k} \Bigr|
    = \frac{|V W|}{1-2|VW|}, \\
    b^{2} &= \frac{|V|^{2}+|W|^{2}-2|VW|}{1-2|VW|},
  \end{align}
  hence
  \begin{equation}
    p = |V|^{2}+|W|^{2}
    \qquad\text{and}\qquad
    q = |VW|.
  \end{equation}
  Because we know the sign of~\(|V|-|W|\), both~\(|V|\) and~\(|W|\) are determined by~\((\aa,b)\),
  hence so are the~\(Z_{k}\) and finally also \(V W\).
  This completes the proof.
\end{proof}

\begin{theorem}
  \label{thm:homeo-Z-Y}
  \(Z\) is equivariantly homeomorphic to~\(Y\).
\end{theorem}

\begin{proof}
  By what we have said at the beginning of this section,
  it is enough to establish a homeomorphism between~\(\Yge\) and~\(\Zge=Q\)
  that respects the orbit type partitions.
  The partition for~\(Q\) is induced by the one for~\(\Sge\) via the Hopf
  fibration~\(\pi\colon\partial Q\to\partial\Sge\). Similarly,
  we see that the partition for~\(\Yge\) is induced by the one for~\(\phi(\Yge)\),
  and \(\phi\colon\partial \Yge\to\phi(\partial \Yge)\) is again the
  Hopf fibration by Lemma~\ref{thm:Q-Hopf}.
  Hence the claim follows
  by observing that the homeomorphism~\(\psi\) is compatible
  with the partitions for~\(\phi(\Yge)\) and~\(\Sge\).
\end{proof}

\section{The real case}
\label{sec:real}

In~\cite[Sec.~6]{FranzPuppe:2008} Franz--Puppe also considered
a ``real version'' of the mutant~\(Z\), which is a compact orientable smooth
manifold~\(\bZ\) of dimension~\(2n\) with a smooth action of the \(2\)-torus~\(G=\{\pm1\}^{n+1}\).
As a module over the polynomial algebra~\(A=H^{*}(BG;\Z_{2})\),
the \(G\)-equivariant cohomology of~\(\bZ\)
is again of the form~\eqref{eq:HTZ}, up to the grading.
By restricting the variables~\(z_{k}\) and~\(Z_{k}\) in the definitions
of~\(X\) and~\(Y\) to real numbers and the action from~\(T\) to~\(G\),
we also obtain compact orientable \(G\)-manifolds~\(\bX\) and~\(\bY\).

\begin{theorem} For the \(G\)-manifolds~\(\bX\),~\(\bY\) and~\(\bZ\) the following holds:
  \begin{enumerate}
  \item \(\bX\) and~\(\bY\) are equivariantly diffeomorphic.
  \item \(\bZ\) is equivariantly homeomorphic to~\(\bX\) and~\(\bY\).
  \end{enumerate}
\end{theorem}

The proof is completely analogous to the ``complex case''.
Alternatively, one can check that the maps constructed in the complex case
commute with the canonical involutions on~\(X\),~\(Y\) and~\(Z\), whose fixed points are
exactly \(\bX\),~\(\bY\) and~\(\bZ\).

The following observation was made
in~\cite[Sec.~7]{FranzPuppe:2008} for~\(\bZ\); 
in the next section we will extend it to the complex case.

\begin{lemma}
  \label{pi-1-bar-Y}
  \(\bY\) is simply connected for~\(n>1\).
\end{lemma}

\begin{proof}
  Since \(\Yge\) is a fundamental domain for the \(G\)-action on~\(\bY\),
  the translates~\(g\Yge\), \(g\in G\), cover \(\bY\).
  For~\(g\ne h\), the intersection~\(g\Yge\cap h\Yge\) is the subspace of~\(g\Yge\)
  obtained by setting \(Z_{k}=0\) for all~\(k\) such that \(g_{k}\ne h_{k}\).
  This intersection contains \(\bY^{G}\ne\emptyset\) and
  it is connected unless \(g=(-1,\dots,-1)\cdot h\), see Remark~\ref{rem:Z-fixed-points}
  (or Lemma~\ref{thm:one-lambda} below for an alternative proof).
  It follows that
  \begin{equation}
   \bigl( g_{1}\Yge\cup\dots\cup g_{k}\Yge \bigr) \cap h\Yge
  \end{equation}
  is connected for~\(k\ge2\) and pairwise distinct \(g_{1}\),~\dots,~\(g_{k}\),~\(h\in G\). 

  We also know that \(\Yge\) is contractible.
  The Seifert--van Kampen theorem now implies
  that if we build up \(\bY\) from~\(\Yge\) by adjoining one~\(g\Yge\) after the other,
  we will always glue together two simply connected spaces along a connected subspace
  and therefore obtain another simply connected space,
  provided that we start with some~\(g\ne(-1,\dots,-1)\) in the first step.
\end{proof}

\section{Connected sums of products of spheres}

Franz--Puppe~\cite{FranzPuppe:2008} have computed
the integer cohomology of the mutants:
In the complex case one gets that \(Z\) has the same integer homology
as the connected sum of products of spheres
\begin{equation}
  \label{eq:sum-complex-case}
  \mathop\#_{k=1}^{\ceil{n/2}} \mathop\#_{\binom{n+1}{k}} S^{n+k}\times S^{2n+1-k},
\end{equation}
where \(\ceil{n/2}\) denotes the least integer greater or equal to~\(n/2\)
(that is, equal to~\(1\) for~\(n=1\) and to~\(n/2\) otherwise).
In the real case, \(\bZ\) has the same integer homology as
\begin{equation}
  \label{eq:sum-real-case}
  \mathop\#_{2^{n}-1} \! S^{n}\times S^{n}.
\end{equation}

\begin{theorem}
  \label{thm:connected-sums}
  In the complex case, \(Y\) is diffeomorphic to~\eqref{eq:sum-complex-case}.
  In the real case, \(\bY\) is diffeomorphic to~\eqref{eq:sum-real-case}.
\end{theorem}

Homeomorphisms of this type were already established in~\cite[Sec.~7]{FranzPuppe:2008}
for~\(Z\) if~\(n\le 2\) and for~\(\bZ\) if~\(n\le 4\).

\begin{proof}
  For~\(n=1\) this is a restatement of the definitions of~\(X\) and~\(\bX\).
  The case~\(n=2\) is a special case of a result
  of Gómez Gutiérrez and López de Medrano~\cite[Main Thm.]{GomezLopezDeMedrano:2014}:
  A unitary change of coordinates replaces \(VW\) by~\((V^{2}+W^{2})/2\);
  the denominator can be absorbed by the~\(\lambda_{k}\)'s.
  The resulting equations are of the form considered in~\cite[p.~240]{GomezLopezDeMedrano:2014},
  where they are written as real polynomials in real variables.
  
  For~\(n=4\) and~\(n=8\) we look at the real case first.
  Then \(\bY\) is of dimension at least \(8\), and simply connected by Lemma~\ref{pi-1-bar-Y}.
  We introduce a new variable~\(Z_{n+1}\) and consider the set of solutions~\(L\)
  of the equations
  \begin{align}
    \label{eq:def-L-1}
    \frac{n}{2}\sum_{k=0}^{n}|Z_{k}|^{2}\lambda_{k} + V W &= 0, \\
    \label{eq:def-L-2}
    \sum_{k=0}^{n+1}|Z_{k}|^{2} + |V|^{2} + |W|^{2} &= 1, \\
    \shortintertext{plus the inequality}
    Z_{n+1} &\ge 0
  \end{align}
  with \(V\),~\(W\in\A\), \(Z_{0}\),~\dots,~\(Z_{n+1}\in\R\) and \(\lambda_{n+1}=\lambda_{0}\).
  Analogous to the discussion
  in~\cite[p.~1505]{GitlerLopezDeMedrano:2013}, we get that \(L\)
  is a manifold with boundary~\(\bY\) and that the inclusion~\(\bY\hookrightarrow L\)
  is a homotopy retraction. Hence \(L\) is also simply connected
  and has the homology of a wedge of \(n\)-spheres.
  As \(L\) is stably parallelizable by construction,
  this implies that it is a \((2n,n)\)-handlebody \cite[Thm.~VIII.4.8]{Kosinski:1993}
  and therefore the boundary-connected sum of copies of~\(D^{n+1}\times S^n\)
  \cite[top of p.~188]{Kosinski:1993}.
  Hence its boundary~\(\bY\) is the connected sum of copies of~\(S^n\times S^n\).
  The number of summands follows from the known Betti numbers.

  We reduce the complex to the real case by another argument from~\cite{GitlerLopezDeMedrano:2013}:
  One can think of the equations for~\(Y\) as being obtained from those for~\(\bY\)
  by introducing another~\(n+1\) real variables~\(Z_{n+1}\),~\dots,~\(Z_{2n+1}\)
  with the same set of~\(\lambda_{k}\)'s,
  \cf~\cite[p.~1506]{GitlerLopezDeMedrano:2013}.
  Now if a simply connected intersection of quadrics of dimension at least~\(5\)
  is isomorphic to a connected sum of
  products of two spheres, then introducing a new variable~\(Z_{k}\) with a~\(\lambda_{k}\)
  that is already present leads again to a simply connected manifold
  which is a connected sum of products of two spheres.
  This is proven in~\cite[Thm.~1.1]{GitlerLopezDeMedrano:2013} for the
  ``diagonal case''
  \begin{equation}
    \sum_{k}|Z_{k}|^{2}\lambda_{k} = 0
    \qquad\text{and}\qquad
    \sum_{k}|Z_{k}|^{2} = 1,
  \end{equation}
  but the argument remains valid if an additional polynomial map
  (\(VW\) in our case) is present from the outset,
  \cf~\cite[Sec.~4.1~\&~p.~254]{GomezLopezDeMedrano:2014}.
  That the new manifold is simply connected follows from the fact that it is the double
  of a simply connected manifold with connected boundary.
  (In the case of~\(\bY\) this would be the manifold~\(L\) mentioned above.)
  Since we already know \(\bY\) to be a connected sum of products of spheres,
  it follows that this is true also for~\(Y\). The number and form of the summands
  is again determined by the homology.
\end{proof}

\section{A generalization}
\label{sec:generalization}

For any~\(m\ge0\) and any collection~\(\bflambda=(\lambda_{1},\dots,\lambda_{m})\in\A^{m}\) one can define
a real algebraic variety~\(Y(\bflambda)\subset\C^{m}\times\A^{2}\) by
\begin{align}
  \label{eq:def-YY-1}
  \sum_{k=1}^{m}|Z_{k}|^{2}\lambda_{k} + V W &= 0, \\
  \label{eq:def-YY-2}
  \sum_{k=1}^{m}|Z_{k}|^{2} + |V|^{2} + |W|^{2} &= 1,
\end{align}
where \(Z_{1}\),~\dots,~\(Z_{m}\in\C\) and \(V\),~\(W\in\A\).
The torus~\(T=(S^{1})^{m}\) acts on~\(Y(\bflambda)\) in the same way as before.
An argument as in the proof of Proposition~\ref{thm:Y-mf} shows
that \(Y(\bflambda)\) is a smooth \(T\)-manifold (of dimension~\(2m+n-1\))
if \(\bflambda\) satisfies the \emph{weak hyperbolicity condition},
\cf~\cite[Sec.~0.1]{GitlerLopezDeMedrano:2013}:
The origin must not be contained in the convex hull of \(n\)~or fewer vectors~\(\lambda_{k}\).
It follows from an equivariant version of the Ehresmann fibration theorem
that if one varies \(\bflambda\),
the equivariant diffeomorphism type of~\(Y(\bflambda)\) does not change
as long as~\(\bflambda\) remains weakly hyperbolic.
In particular, if \(m=n+1\), then \(Y(\bflambda)\) is equivariantly diffeomorphic to
the~\(Y\) defined in the introduction
whenever \(\bflambda\) spans an \(n\)-simplex containing the origin in its interior.
(In particular, the factor~``\(n/2\)'' appearing in~\eqref{eq:def-Y-1} is irrelevant.)

The following observation strengthens and generalizes Remark~\ref{rem:Z-fixed-points}.

\begin{lemma}
  \label{thm:one-lambda}
  Let \(\bflambda\) be weakly hyperbolic.
  If the origin is not contained in the convex hull of~\(\bflambda\),
  then there is an equivariant diffeomorphism
  \begin{equation*}
    Y(\bflambda) \cong S^{2m}\times S^{n-1}
  \end{equation*}
  where \(T\) acts in the standard fashion on the one-point compactification~\(S^{2m}\) of~\(\C^{m}\).
\end{lemma}

\begin{proof}
  One can deform all~\(\lambda_{k}\) to the same vector, say \(1/2\in\R\).
  If \(n=m=1\), then \(Y(\bflambda)\) is the fixed point set of a circle
  in the~\(Y\) defined by~\eqref{eq:def-Y-1} and~\eqref{eq:def-Y-2}.
  The claimed equivariant diffeomorphism with~\(S^{2}\times S^{0}\) thus
  follows from the equivariant diffeomorphism between \(Y\) and~\(X=S^{2}\times S^{2}\).
  Looking at the definition of the latter diffeomorphism
  in Theorem~\ref{thm:diffeo-X-Y}\,\eqref{p2},
  we see that it extends to the case~\(m>1\).
  (Just replace \(Z_{1}^{2}\) by~\(Z_{1}^{2}+\dots+Z_{m}^{2}\).)
  We also note that the non-trivial involution~\(\tau\) of~\(S^{0}\)
  corresponds to swapping \(V\) and~\(W\) (which cannot be equal).

  Now we compare the case~\(n=1\) with the general case. We write the two
  \(T\)-manifolds as~\(Y_{\R}(\bflambda)\) and \(Y_{\A}(\bflambda)\), and \(S_{\A}\subset\A\)
  for the elements of norm~\(1\). Because all (identical)~\(\lambda_{k}\)
  are real, the \(T\)-equivariant map
  \begin{equation}
    Y_{\R}(\bflambda)\times S_{\A}\to Y_{\A}(\bflambda),
    \qquad
    (\ZZ,V,W,a) \mapsto (\ZZ,V a^{-1},a W)
  \end{equation}
  is surjective, and it identifies \((\ZZ,V,W,a)\) with~\((\ZZ,W,V,-a)\).
  Hence
  \begin{equation}
    Y_{\A}(\bflambda) \cong Y_{\R}(\bflambda)\times_{\tau} S_{\A}
    \cong (S^{2m}\times S^{0}) \times_{\tau} S^{n-1} = S^{2m}\times S^{n-1}.
    \qedhere
  \end{equation}
\end{proof}

We can now show that the \(T\)-manifolds~\(Y(\bflambda)\) do not
produce examples of higher non-free syzygies in equivariant cohomology:

\begin{proposition}
  \label{thm:HT-Y-lambda}
  Let \(\bflambda\in\A^{m}\) be weakly hyperbolic, and assume \(n>1\).
  \begin{enumerate}
  \item If the origin is not contained in the convex hull of~\(\bflambda\),
    then \(\HT^{*}(Y(\bflambda);\R)\) is free over~\(A=H^{*}(BT;\R)\).
  \item Otherwise, \(\HT^{*}(Y(\bflambda);\R)\) is torsion-free over~\(A\), but not reflexive.
  \end{enumerate}
\end{proposition}

To verify that \(\HT^{*}(Y)\) is torsion-free
we will use the quotient criterion of~\cite[Prop.~7.16]{Franz:orbits3}
(which in this case is essentially a reformulation of the equivariant injectivity criterion
of~\cite[Thm.~5.9]{GoertschesRollenske:2011}).
We can state it as follows: Let \(X\) be a compact manifold with an effective \(T\)-action having fixed points and
such that all isotropy groups are connected. Assume moreover that the non-free part
of the action is the union of codimension-\(2\) submanifolds, so that the quotient~\(X/T\)
is a manifold with corners, \cf~\cite[pp.~9 \&~12]{Franz:orbits3}.
Note that the minimal faces of~\(X/T\) correspond to the components of~\(X^{T}\).
For such an~\(X\) the equivariant cohomology \(\HT^{*}(X)\) is torsion-free over \(A=H^{*}(BT;\R)\)
if and only if for every face~\(P\) of~\(X/T\) not corresponding to a fixed point component
the restriction map
\begin{equation}
  \label{eq:quotient-criterion}
  H^{*}(P;\R) \to \bigoplus H^{*}(Q;\R)
\end{equation}
is injective, where \(Q\) runs through the facets of~\(P\).

\begin{proof}
  For convenience, we write \(Y=Y(\bflambda)\).
  Assume first that the convex hull of~\(\bflambda\) does not contain the origin.
  It then follows from Lemma~\ref{thm:one-lambda} that
  \(Y^{T} \cong S^{0}\times S^{n-1}\)
  has the same Betti sum as~\(Y\), which implies that
  \(\HT^{*}(Y;\R)\) is free over~\(A\), \cf~\cite[Thm.~3.10.4]{AlldayPuppe:1993}.

  Now assume that the origin is contained in the convex hull of~\(\bflambda\), say
  \begin{equation}
    \sum_{k=1}^{m} c_{k}\lambda_{k} = 0
    \qquad\text{and}\qquad
    \sum_{k=1}^{m} c_{k} = 1
  \end{equation}
  for some~\(c\in\Rge^{m}\).

  We start by observing that the orbit space~\(Y_{+}\)~itself is contractible:
  For a non-zero element~\(a\in\A\), let us write
  \(N(a) = a/|a|\) for its normalization, and also
  \(\sqrt{\ZZ}=(\sqrt{Z_{1}},\dots,\sqrt{Z_{m}})\) for~\(\ZZ\in\Rge^{m}\).
  Then the map~\(h\colon[0,1]\times Y_{+}\to Y_{+}\),
  \begin{equation}
    h(t,\ZZ,V,W) = N \bigl( \sqrt{t\ZZ+(1-t)c},\sqrt{t}\,V,\sqrt{t}\,W\bigr)
  \end{equation}
  is well-defined and contracts \(Y_{+}\) to~\((c,0,0)\).

  More generally, let \(P\) be any face of~\(Y_{+}\) not corresponding to a component of~\(Y^{T}\),
  and let \(L\subset T\) be the corresponding isotropy subgroup with quotient~\(M=T/L\).
  Then \(L\) is a coordinate subtorus
  of~\(T\), corresponding to some subset~\(K\subset\{1,\dots,m\}\), and \(Y^{L}\) is of the
  form~\(Y(\bflambda')\), where \(\bflambda'\) is formed by the~\(\lambda_{k}\) with~\(k\notin K\).

  Hence by the same token as above, the orbit space~\(P=Y^{L}/M\)
  is contractible as long as the origin is contained in the convex hull of~\(\bflambda'\),
  and \(H^{*}(P;\R)\) is acyclic in this case.
  Because \(P\) has at least one facet, the map~\eqref{eq:quotient-criterion} is injective.

  If the origin is not contained in~\(\bflambda'\), then we already know
  \(H_{M}^{*}(Y^{L})\) to be free over~\(H^{*}(BM;\R)\).
  Applying the quotient criterion to~\(Y^{L}\), we see that
  \eqref{eq:quotient-criterion} is injective for such a~\(P=Y^{L}/M\) as well.
  We conclude that \(\HT^{*}(Y;\R)\) is torsion-free.

  We finally show that \(\HT^{*}(Y;\R)\) is not reflexive.
  It it were, then so would be \(H_{M}^{*}(Y^{L};\R)\) over~\(H^{*}(BM;\R)\)
  for any subtorus~\(L\subset T\) with quotient~\(M=T/L\) \cite[Cor.~2.2]{Franz:orbits3}.
  In particular, it would be true for any~\(L\) such that \(\bflambda'\)
  (as defined above) has length~\(n+1\) and such that the corresponding convex hull
  contains the origin. But in this case \(Y^{L}\) is equivariantly diffeomorphic
  to the~\(Y\) defined by~\eqref{eq:def-Y-1} and~\eqref{eq:def-Y-2},
  and its equivariant cohomology~\eqref{eq:HTZ} is not reflexive
  (because 
  the double dual of the \(A\)-module~\(\m\) is isomorphic to~\(A\), not to~\(\m\)).
  Contradiction.
\end{proof}

\begin{remark}
  More generally, for any~\(s\ge1\) one can consider the real algebraic variety~\(Y(\bflambda,s)\)
  given by
  \begin{align}
    \sum_{k=1}^{m}|Z_{k}|^{2}\lambda_{k} + \sum_{l=1}^{s} V_{l} W_{l} &= 0, \\
    \sum_{k=1}^{m}|Z_{k}|^{2} + \sum_{l=1}^{s} \bigl(|V_{l}|^{2} + |W_{l}|^{2}\bigr) &= 1,
  \end{align}
  where now \(V_{1}\),~\ldots,~\(V_{s}\),~\(W_{1}\),~\ldots,~\(W_{s}\in\A\).
  Again, \(Y(\bflambda,s)\) is smooth if \(\bflambda\) is weakly hyperbolic.
  
  For~\(n=2\)
  these manifolds were studied in~\cite{GomezLopezDeMedrano:2014},
  and the results there include the following:
  If \(m=0\), then \(Y(\bflambda,s)=V_{2s,2}\) is the Stiefel manifold of orthogonal \(2\)-frames
  in~\(\C^{s}=\R^{2s}\) (whose Betti sum equals \(4\)).
  Otherwise, if \(\bflambda\) does not contain the origin, then \(Y(\bflambda,s)\)
  is diffeomorphic to~\(S^{m+2s-2}\times S^{2s-1}\).
  If \(m=3\) and \(\bflambda\) does contain \(0\), then
  \begin{equation}
    Y(\bflambda,s) \cong
    \mathop\#_{3} S^{2s}\times S^{2s},
  \end{equation}
  which generalizes \eqref{eq:sum-real-case}.

  Still assuming \(n=2\),
  the conclusions of Proposition~\ref{thm:HT-Y-lambda}
  continue to hold for the equivariant cohomology of~\(Y(\bflambda,s)\),
  and the proof proceeds in the same way as before.
  For \(m=3\) and \(\bflambda\) containing the origin,
  we note that \(\HT^{*}(Y(\bflambda,s))\) cannot be reflexive
  because otherwise it would already be free over~\(A\)
  \cite[Cor.~1.4]{AlldayFranzPuppe:orbits1},
  which would contradict the fact that the Betti sums of~\(Y(\bflambda,s)\) and \(Y(\bflambda,s)^{T}=V_{2s,2}\)
  differ.
  
  For~\(n>2\) and \(s\ge2\) it is unclear
  whether the \(A\)-module~\(\HT^{*}(Y(\bflambda,s))\) can be reflexive
  without being free. If this is possible, then it happens already
  for the choice of~\(\bflambda\) given by~\eqref{eq:cond-lambda}.
\end{remark}

\end{document}